\documentclass[a4paper,11pt]{amsart}
\usepackage{amssymb,amsmath,amsfonts,amscd,euscript,accents,mathtools}
\usepackage{accents}
\newcommand{\ubar}[1]{\underaccent{\bar}{#1}}
\usepackage[backref=page]{hyperref}
\usepackage{color}
\usepackage[svgnames, x11names]{xcolor}
\definecolor{pink}{rgb}{1, 0, 1}

\usepackage[shortlabels]{enumitem}
\usepackage{todonotes}

\usepackage{tikz, tikz-3dplot, pgfplots, graphicx}
\usepackage{tikz-cd}
\usetikzlibrary{arrows,shapes,automata,backgrounds,petri,positioning, graphs}

\numberwithin{equation}{section}

\newtheorem{theoremA}{Theorem}

\newtheorem{thm}{Theorem}[section]
\newtheorem{prop}[thm]{Proposition}
\newtheorem{lem}[thm]{Lemma}
\newtheorem{cor}[thm]{Corollary}
\newtheorem{rem}[thm]{Remark}
\newtheorem{definition}[thm]{Definition}
\newtheorem{example}[thm]{Example}
\newtheorem{dfn}[thm]{Definition}

\newtheorem{conj}[thm]{{\bf Conjecture}}

\newcommand{\J}{\mathcal{J}}

\newcommand{\bC}{{\mathbb C}}

\newcommand{\bZ}{{\mathbb Z}}
\newcommand{\bN}{{\mathbb N}}

\newcommand{\fg}{{\mathfrak g}}

\newcommand{\addots}{\text{\reflectbox{$\ddots$}}}

\DeclareMathOperator{\Aut}{Aut}

\DeclareMathOperator{\Span}{Span}

\pgfplotsset{compat=1.18} 
\begin{document}


\title[Symplectic Grassmannians and cyclic quivers]
{Symplectic Grassmannians and cyclic quivers}

\author{Evgeny Feigin}
\address{E. Feigin:\newline
School of Mathematical Sciences, Tel Aviv University, Tel Aviv
69978, Israel
}
\email{evgfeig@gmail.com}
\author{Martina Lanini}
\address{M. Lanini:\newline Dipartimento di Matematica\\ Universit\`a di Roma ``Tor Vergata'',  Via della Ricerca Scientifica 1, I-00133 Rome, Italy}
\email{lanini@mat.uniroma2.it}
\author{Matteo Micheli}
\address{M. Micheli: \newline Dipartimento di Matematica "Guido Castelnuovo"\\ Sapienza Universit\`a di Roma, Piazzale Aldo Moro 5, I-00185 Rome, Italy}
\email{matteo.micheli@uniroma1.it}
\author{Alexander P\"utz}
\address{A. P\"utz:\newline
Institute of Mathematics\\
University Paderborn\\
Warburger Str. 100\\\newline
D-33098 Paderborn\\
Germany}
\email{alexander.puetz@math.uni-paderborn.de}


\begin{abstract}
The goal of this paper is to extend the quiver Grassmannian description of certain
degenerations of Grassmann varieties to the symplectic case.
We introduce a symplectic version of quiver Grassmannians studied in our previous
papers and prove a number of results on these projective algebraic varieties. 
First, we construct a cellular decomposition of the symplectic quiver Grassmannians
in question and develop combinatorics needed to compute Euler characteristics and Poincar\'e polynomials. Second, we show that the number of irreducible components of our varieties coincides with the Euler characteristic of the classical symplectic Grassmannians. Third, we describe the automorphism groups of the underlying symplectic quiver representations and show that the cells are the orbits of this group. Lastly, we provide an embedding into the affine flag varieties for the affine symplectic group.
\end{abstract}

\maketitle

\section{Introduction}
The classical Grassmann varieties $\mathrm{Gr}(k,n)$ admit a flat
degeneration into certain reducible algebraic varieties $X(k,n)$. The construction emerged from arithmetics as local models of Shimura varieties \cite{Go01,Ga01,PRS13}, but is also
very natural from the point of view of complex algebraic geometry \cite{Kn08,Zho19}. In \cite{FLP22,FLP23a,FLP23b}
the quiver Grassmannians approach describing these degenerations was
developed. More precisely, the authors considered certain modules $U_{[n]}$ over the cyclic quivers $\Delta_n$ such that $X(k,n)$ is isomorphic to the quiver Grassmannian 
of subrepresentations of $U_{[n]}$ of dimension $k$ at each vertex. 
This paper is the first step towards the extension of the quiver approach to the symplectic case (see \cite{BF19,BF20,BC22,BCFF24} for the non cyclic type $A$ case).

The varieties $X(k,n)$ admit many nice properties. Having a quiver Grassmannians realization at hand, one is able to use various techniques from the algebraic, combinatorial and
geometric theories of quivers in order to study various structures related to $X(k,n)$. In particular, one can use the automorphism groups of the underlying 
representations in order to describe cellular decompositions and to establish a link with the positroid decomposition of totally nonnegative Grassmannians. Looking at the side of local models of Shimura varieties, one gets generalizations for other classical groups \cite{Go03,Pa18,PZ22} (we note that only the Lagrangian case was considered).
Hence it is natural to ask for the quiver realization with a bilinear form added as an extra piece of data. Our  goal  is to 
initiate the study of the corresponding quiver Grassmannians.

In this paper we deal only with the symplectic form on a $2n$ dimensional complex vector space. 
The symplectic Grassmannians $\mathrm{Gr}(k,2n)^{sp}$ have been extensively studied in the literature  (see e.g. \cite{DC,FH}). For each $k=1,\dots,n$ we define a degeneration $X(k,2n)^{sp}$ of the corresponding symplectic Grassmannian as follows.
The symplectic form on the $2n$ dimensional complex vector space as above induces a non degenerate skew-symmetric form on the representation space $U_{[2n]}$. Note that this is a symplectic representation in the sense of Derksen-Weyman \cite{DW}. Then
$X(k,2n)^{sp}$ is defined as a subvariety of $X(k,2n)$ consisting of self-orthogonal subrepresentations.
Our first theorem is as follows.

\begin{theoremA} (Theorem~\ref{thm:top-dim-cells})
The intersection of a cell of $X(k,2n)$ with $X(k,2n)^{sp}$
is either empty or an affine cell. The top dimensional cells of $X(k,2n)^{sp}$ are obtained from top dimensional cells of $X(k,2n)$.
The dimension of $X(k,2n)^{sp}$ is equal to the dimension of the classical symplectic Grassmannian $\mathrm{Gr}(k,2n)^{sp}$ and the number of
irreducible components of $X(k,2n)^{sp}$ is equal to the Euler characteristic of $\mathrm{Gr}(k,2n)^{sp}$.
\end{theoremA}

Let $G$ be the automorphism group of $U_{[2n]}$ and let $G^{sp}$ be its subgroup
preserving the symplectic form. It was shown in \cite{FLP22,FLP23a,FLP23b} that $G$ is a degeneration of the general linear group $GL_{2n}$ and the action of $G$ can be used to describe various properties of the varieties $X(k,2n)$.
We show that similar properties hold true in the symplectic case as well.

\begin{theoremA} (Theorem~\ref{thm:group-orbits-are-cells})
The group $G^{sp}$ is a degeneration of the classical group $Sp_{2n}$; in particular, $\dim G^{sp}=2n^2+n$. The group $G^{sp}$ acts on $X(k,2n)^{sp}$ with a finite number of orbits, the orbits are affine cells and are naturally labeled by certain combinatorial gadgets that we call  symplectic juggling patterns.
\end{theoremA}

Recall that the varieties $X(k,2n)$ admit embeddings into the affine  flag  varieties of type $A$ \cite{HR20,Zho19,Zhu19}.
These embeddings have  simple and transparent descriptions in terms of quiver representations via the lattice realization of affine Grassmannians and affine flag varieties (\cite{FLP23a,HZ23-1,HZ23-2}).
We show that one can extend the type $A$ picture to the symplectic case. More precisely, we prove 
the following theorem (it is known for $k=n$ from the arithmetic side by \cite{PRS13}).

\begin{theoremA}
(Theorem~\ref{thm:emb-affine-symp-flag}) The varieties $X(k,2n)^{sp}$ admit an embedding into the flag variety for the affine symplectic group $Sp_{2n}$. The image of the embedding is a union of Schubert varieties. 
\end{theoremA}

We put forward two conjectures:
\begin{itemize}
    \item The degeneration from the symplectic Grassmannian $\mathrm{Gr}(k,2n)^{sp}$ to \newline $X(k,2n)^{sp}$ is flat (see Section~\ref{sec:degeneration} for more detail);
    \item The poset structure on the set  of cells of $X(k,2n)^{sp}$ is induced from the corresponding poset of cells of $X(k,2n)$ (see Conjecture~\ref{conj:induced-partial-order}). 
\end{itemize}

Finally, let us mention several possible further directions.
First, it would be interesting to study the projections $X_S(k,2n)^{sp}$ of the varieties
$X(k,2n)^{sp}$ to the product of Grassmannians corresponding to a
subset $S$ of vertices of the quiver $\Delta_{2n}$. Second, the varieties
$X(k,2n)$ admit certain $\omega$-generalizations $X(k,2n,\omega)$ \cite{FLP23a}, it would be interesting to develop the corresponding symplectic story. Third, one is interested in the projections $X_S(k,2n,\omega)^{sp}$. In all these cases one expects a link 
to a flat degeneration from Schubert varieties in affine Grassmannians to unions of Schubert varieties in affine flag varieties (see \cite{Ga01,Zho19,FLP23b}). Fourth, the quiver Grassmannian realization allows one to use quiver techniques for the construction of resolutions of singularities \cite{FF13,CFR13,PR23}.
It is natural to ask for a similar construction in the symplectic  case (see \cite{FFL14} for a special case of such construction).
Fifth, it is natural to expect a link between the topology of 
the varieties $X(k,2n)^{sp}$ and the symplectic version of the totally nonnegative Grassmannians (in the spirit of \cite{FLP22}, see also \cite{Kar18,W05,Pos06}). Lastly, it is natural to study the orthogonal case
(both even and odd) and the case of odd symplectic Grassmannians \cite{M07,Pr88}.

\vspace{2mm}
Our paper is organized as follows. In Section \ref{sec:background} we provide the needed background; in particular, we introduce here the poset $JP(k,n)$ of $(k,n)$-juggling patterns and the juggling variety $X(k,n)$, as well as their symplectic analogues $JP(k,2n)^{sp}$ and $X(k,2n)^{sp}$. 
Section \ref{sec:automorphisms} is about reductive group actions: we recall the action of an appropriate degeneration of $GL_{2n}$ on $X(k,2n)$, we define an involution on such a group and look at both its fixed point subgroup $G^{sp}$ (the symplectic automorphism group) and at the corresponding Lie algebra.
The symplectic automorphism group acts on $X(k,2n)^{sp}$ and we devote the rest of this section to the investigation of its orbits. More precisely, we show that they are parameterized by symplectic juggling patterns and that they are all affine spaces. 
In Section \ref{sec:orbits} we deepen the study of the $G^{sp}$-orbits on $X(k,2n)^{sp}$. More specifically, we connect the orbit closure inclusion relation
to certain combinatorial moves on juggling patterns called mutations, and exploit them to determine the cell dimension.
Moreover, we provide an explicit counting of the top dimensional cells, and hence determine the dimension of our variety of interest. 
Section \ref{sec:affine} deals with the affine Grassmannians and flag variety: after recalling the embedding of $X(k,2n)$ into the $GL_{2n}$ affine flag variety, we show that $X(k,2n)^{sp}$ can be embedded into the $Sp_{2n}$-affine flag variety. In Appendix \ref{sec:appendix} we exhibit Poincar\'e polynomials and Euler characteristics of $X(k,2n)$ and its symplectic analogue for $n=1,2,3,4$ and $k\leq n$.

\section{Background}
\label{sec:background}
\subsection{Juggling patterns and juggling variety}
First we lay down some notation: 
\begin{itemize}
    \item for $n \le m \in \bN^+$, we will write $[n]$ for the set $\{1, 2, \dots, n\}$, $[n,m]$ for $\{n, n+1, \dots, m\}$ and $\binom{[m]}{n}$ for the set of subsets of $[m]$ with cardinality $n$;
    \item $e_1, e_2, \dots, e_n$ denote the elements of the standard basis for $\bC^n$. Given a subset $I$ of $[n]$, the coordinate subspace $V_I$ is defined as $\Span \{e_i \, \vert \, i \in I\}$;
    \item if $g$ is an invertible matrix, $g^{-t}$ will be notation for $(g^{-1})^t$;
    \item the ring of integers modulo $m$ will be denoted by $\bZ_m$.
\end{itemize}
We denote with $\Delta_n$ the equioriented quiver of type $\widetilde{A}$ on $n$ vertices.

\begin{center}
    \begin{tikzpicture}
     \begin{scope}[every node/.style={circle, draw=black!100, 
     very thin,minimum size=1mm}]
    \node (0) at (0,2.2) {};
    \node (1) at (2,1) {};
    \node (2) at (2,-1) {};
    \node (3) at (0,-2.2) {};
    \node (4) at (-2,-1) {};
    \node (5) at (-2,1) {};
\end{scope}

\begin{scope}[every edge/.style= 
              {draw=black,thick}]
\path [->] (0) edge[bend left=20] (1);
\path [->] (1) edge[bend left=20] (2);
\path [->] (2) edge[bend left=20] (3);
\path [->] (3) edge[bend left=20] (4);
\path [->] (4) edge[bend left=20] (5);
\path [->] (5) edge[bend left=20] (0);
\end{scope}
    \end{tikzpicture}
     
    The quiver $\Delta_6$.
    \end{center}
    We label its vertices with the integers modulo $n$, so that the arrows are of the form $i \longrightarrow i+1$.

\begin{definition}
    Let $k \le n$ be natural numbers; we define the projective variety
    \[X(k,n) \coloneqq \Big\{V = (V_i)_{i \in \bZ_n} \in \prod_{i \in \bZ_n} Gr(k,\bC^n) \, \vert \, \tau_1(V_i) \subseteq V_{i+1} \, \forall \, i \in \bZ_n\Big\}\]
    where the endomorphism $\tau_1$ of $\mathbb{C}^{n}$ is given by $\tau_1(e_i) = e_{i+1}$ for $i<n$ and $\tau_1(e_n) =0$.
\end{definition}
  Notice that 
  $X(k,n)$ is a quiver Grassmannian for $\Delta_n$ 
  (see \cite{FLP22} for more detail): consider the nilpotent $\Delta_n$-representation $U_{[n]}$ which has $\bC^n$ over each vertex and $\tau_1$ over each arrow. If we denote by $\ubar{k} \in \bN^n$ the dimension vector whose entries are all equal to  $k\in \bN$, we have that $X(k,n)$ is the locus of subrepresentations of $U_{[n]}$ of dimension vector $\ubar{k}$, i.e.
\[X(k,n) = Gr(\ubar{k}, U_{[n]}) \, .\]

\begin{center}
    \begin{tikzpicture}
     \begin{scope}[every node/.style={circle,
     very thin,minimum size=1mm}]
    \node (0) at (0,2.2) {$\bC^6$};
    \node (1) at (2,1) {$\bC^6$};
    \node (2) at (2,-1) {$\bC^6$};
    \node (3) at (0,-2.2) {$\bC^6$};
    \node (4) at (-2,-1) {$\bC^6$};
    \node (5) at (-2,1) {$\bC^6$};
\end{scope}

\begin{scope}[every edge/.style= 
              {draw=black,thick}]
\path [->] (0) edge[bend left=20, "$\tau_1$"] (1);
\path [->] (1) edge[bend left=20, "$\tau_1$"] (2);
\path [->] (2) edge[bend left=20, "$\tau_1$"] (3);
\path [->] (3) edge[bend left=20, "$\tau_1$"] (4);
\path [->] (4) edge[bend left=20, "$\tau_1$"] (5);
\path [->] (5) edge[bend left=20, "$\tau_1$"] (0);
\end{scope}
    \end{tikzpicture}
    
    The representation $U_{[6]}$.
    \end{center}

\begin{definition}
    Let $k \le n$ be natural numbers. A $(k,n)$-juggling pattern is a collection $\J = (J_i)_{i \in \bZ_n}$ of cardinality $k$ subsets of $[n]$, such that $j \in J_i$ implies $j+1 \in J_{i+1}$, for all $i \in \bZ_n$ and  all $j \in [n-1]$. The set of $(k,n)$-juggling patterns is denoted by $JP(k,n)$.
\end{definition}

\begin{rem}\label{rem:coordinate-sub-reps}
    Observe that for each $(k,n)$-juggling pattern $\J = (J_i)_i$, the collection of coordinate vector spaces $p_\J \coloneqq (V_{J_i})_i$ is a point in $X(k,n)$.
\end{rem}

In \cite{FLP22} a different, but isomorphic, $\Delta_n$-representation is chosen: the map on the arrow $i \longrightarrow i+1$ is the projection along the $i$-th standard basis vector. The corresponding quiver Grassmannians are again isomorphic to ours, and the combinatorial gadgets resulting from this representation are called $(k,n)$-Grassmann necklaces. One can produce a juggling pattern from a Grassmann necklace and viceversa \cite[2.2]{FLP22}, and all results we recall that use one family of objects still hold for the other. The $(k,n)$-juggling patterns admit certain $\omega$-generalizations (\cite{FLP23a}) which do not naturally exist for Grassmann necklaces. With this generalization of the symplectic setting in mind we decided to work with juggling patterns instead of Grassmann necklaces.

\medskip

For all $k \in [0,n]$, the set $JP(k,n)$ can be equipped with the following partial order: $\J \le \J'$ if and only if $J_i \ge J'_i$ for all $i \in \bZ/n\bZ$, where two sets $A,B \in \binom{[n]}{k}$ satisfy $A \le B$ if, when written in increasing order $A = \{a_1 < a_2 < \dots < a_k\}$, $B = \{b_1 < b_2 < \dots < b_k\}$, one has $a_i \le b_i$ for all $i \in [k]$.

\noindent Let $G \coloneqq \Aut_{\Delta_n}(U_{[n]})$ be the automorphism group of the $\Delta_n$-representation $U_{[n]}$. Then $G$ acts on all quiver Grassmannians for $U_{[n]}$, in particular on $X(k,n)$. Each $G$-orbit in $X(k,n)$ contains exactly one point $p_\J$ labeled by a juggling pattern $\J$ as in Remark~\ref{rem:coordinate-sub-reps}. Such an orbit is an affine cell by \cite[Theorem~1]{FLP22} and we denote it by $C_\J$. Thus the closure inclusion order on the set of $G$-orbits induces a partial order on $JP(k,n)$, which coincides with the combinatorial order described above \cite[Corollary~4.7]{FLP23a}.

\subsection{Symplectic conditions on the juggling variety}
To introduce symplectic conditions on such varieties we need the dimension of the vector spaces to be even, so we work in $X(k,2n)$. We equip $\bC^{2n}$ with the symplectic form
\[(v,w) = \sum_{i=1}^n (-1)^{i+1} \cdot v_i \cdot w_{2n-i+1} \, \, ,\]
where $v=\sum_{i=1}^n v_i e_i$ and $w=\sum_{i=1}^n w_ie_i$.
In other words, this is the symplectic form whose Gram matrix $\Omega$ in the standard basis has zeros everywhere except on the antidiagonal, where it has alternating 1s and $-1$s, with a 1 on the upper right corner.
\[\Omega = \begin{pmatrix}
    0 & 0 & \cdots & 0 & 1 \\
    0 & 0 & \cdots & -1 & 0 \\
    \vdots & \vdots & \addots & \vdots & \vdots \\
    0 & 1 & \cdots & 0 & 0 \\
    -1 & 0 & \cdots & 0 & 0
\end{pmatrix}\]
When $n$ is fixed, we write $\Tilde{i}$ for $2n-i+1$ for brevity. This way we have $(e_i, e_j) = (-1)^{\Tilde{i}} \delta_{\Tilde{i}j}$.

\begin{prop}\label{tau}
    Given $k, n$ natural numbers with $0 \le k \le 2n$, the map
    \begin{align*}
        \sigma \colon X(k,2n) &\longrightarrow X(2n-k,2n) \\
        (V_i)_i &\longmapsto (V_{-i}^\perp)_i
    \end{align*}
    is well defined.
\end{prop}
For $V \in X(k,n)$ and $W \in X(\ell,n)$ with $k \le \ell$, we write $V \subseteq W$ whenever $V_i \subseteq W_i$ for all $i \in \bZ_n$.
\begin{proof}[Proof of Proposition~\ref{tau}]
    To start with, observe that $-\Omega = \Omega^t = \Omega^{-1}$. Moreover, notice  that, as matrices in the standard basis, $-\Omega \cdot \tau_1^t \cdot \Omega^{-1} = \Omega \tau_1^t \Omega = \tau_1$. From linear algebra we obtain: given $W \subseteq \bC^{2n}$ a subspace and $M$ a $2n \times 2n$ matrix, one has $\Omega M^t \Omega \bigl( (M \cdot W)^\perp \bigr) \subseteq W^\perp$. Now if $V = (V_i)_i$ is a point in $X(k,2n)$, its vector spaces satisfy $\tau_1 \cdot V_i \subseteq V_{i+1}$ for all $i$. Taking the orthogonal subspaces we get $V_{i+1}^\perp \subseteq (\tau_1 \cdot V_i)^\perp$, then applying $\tau_1$ to both sides we find
    \[ \tau_1( V_{i+1}^\perp ) \subseteq  \tau_1 \bigl( ( \tau_1\cdot V_i )^\perp \bigr) = \Omega \tau_1^t \Omega \bigl( ( \tau_1\cdot V_i )^\perp \bigr) \subseteq V^\perp_{i} \, \, .\]
    That is, $\tau_1(\sigma V)_{-i-1}\subseteq (\sigma V)_{-i}$ for any $i\in\mathbb{Z}_{2n}$ and hence we obtain
    $\sigma V \in X(2n-k,2n)$.
\end{proof}

\begin{rem}
    The composition
    \[X(k,2n) \overset{\sigma}{\longrightarrow} X(2n-k,2n) \overset{\sigma}{\longrightarrow} X(k,2n)\]
    is the identity map.
\end{rem}
\subsection{The main object}
Now we introduce the main object of this paper:
\begin{definition}
    Let $k \le n$. Then
    \[X(k,2n)^{sp} \coloneqq \{V=(V_i)_i \in X(k,2n) \, \vert \, V \subseteq \sigma V\} \, .\]
    This is the subvariety of isotropic, or symplectic, points in $X(k,2n)$.
\end{definition}

\begin{rem}\label{iso-coiso}
    If instead $n \le k \le 2n$, we can define coisotropic points in $X(k,2n)$ as those that satisfy $V \supseteq \sigma V$. Then the corresponding subvariety is isomorphic to $X(2n-k,2n)^{sp}$ via $\sigma$.
\end{rem}

  For a subset $I \in \binom{[2n]}{k}$, let $RI \coloneqq [2n]$ \textbackslash $\{\tilde{i} \, \vert \, i \in I\} \in \binom{[2n]}{2n-k}$. Notice that $V_{RI}$, the coordinate subspace corresponding to $RI$, coincides with $V_I^\perp$. If $I \subseteq RI$ we say that $I$ is isotropic, or coisotropic if $I \supseteq RI$; in either case we say that $I$ is symplectic.

  Next, we extend $R \colon \binom{[2n]}{k} \longrightarrow \binom{[2n]}{2n-k}$ to the poset of juggling patterns: for $\J=(J_i)\in JP(k,2n)$, we define $R\J$ as the tuple of sets $(R\J)_i = R(J_{-i})$; this is a $(2n-k,2n)$-juggling pattern, since $\sigma(p_\J) \in X(2n-k,2n)$ and it has a coordinate vector space $V_{R(J_{-i})} = V_{(R\J)_i}$ on vertex $i$. Once again, $R$ is a bijection and $R(R\J) = \J$. Given two juggling patterns $\J \in JP(k,n)$ and $\J' \in JP(\ell,n)$ with $k \le \ell$, we write $\J \subseteq \J'$ if $p_\J \subseteq p_{\J'}$, that is, if $J_i \subseteq J'_i$ for all $i$.

\begin{definition}
    Let $k \le n$; a $(k,2n)$-juggling pattern $\J$ is isotropic, or symplectic, if $p_\J \in X(k,2n)^{sp}$ or, equivalently, if $\J \subseteq R\J$. By $JP(k,2n)^{sp}$ we denote the set of $(k,2n)$-symplectic juggling patterns.
\end{definition}

\begin{example}
    Let $k=n=2$.
    
    \noindent Then $\mathcal{J}=(J_0=\{3,4\},J_1=\{2,4\}, J_2=\{2,3\}, J_3=\{3,4\})$ is a $(2,4)$-juggling pattern, but it is not symplectic, as $J_2\not\subseteq RJ_2=\{1,4\}$. Instead, $\mathcal{J}'=(J'_0=\{3,4\},J'_1=\{2,4\}, J'_2=\{3,4\}, J'_3=\{2,4\})\in JP(2,4)^{sp}$.
\end{example}


\begin{rem}
    A $(k,2n)$-juggling pattern $\J$ is maximal in the partial order if for all $i$, $2n \in J_i$ implies $1 \in J_{i+1}$ \cite[Remark 4.12]{FLP22}. Then a maximal $\J$ is symplectic if and only if $J_0$ is symplectic, or equivalently if and only if $J_n$ is symplectic. 
\end{rem}

\begin{prop}\label{k1case}
    For any $n \ge 1$, $X(1,2n)^{sp}= X(1,2n)$.
\end{prop}

\begin{proof}
First we prove that any $(1,2n)$-juggling pattern is isotropic: assume there exists a non-symplectic $(1,2n)$-juggling pattern $\J$, with $J_i = \{x\}$ and $J_{-i} = \{\Tilde{x}\} = \{2n-x+1\}$ for some $x \in [2n]$ and $i \in \bZ_{2n}$. Let $m$ be the number of arrows on the minimal path from $-i$ to $i$, which is equivalent to $2i$ modulo $2n$ and therefore even. Since $x$ and $\Tilde{x}$ have different parity, we see that both between $i$ and $-i-1$ and between $-i$ and $i-1$ there must be a vertex $a$ such that $J_a = \{2n\}$; this means $x+m > 2n$ and $\Tilde{x} + 2n-m > 2n$. These two inequalities are incompatible, thus we have found a contradiction.

Next we show that if $\J$ is a maximal $(1,2n)$-juggling pattern, then the whole orbit $C_\J = G \cdot p_\J$ consists of symplectic points: let $i$ be the vertex with $J_i = \{1\}$, so we have $J_{i+m} = \{i+m\}$ for $m \in [0,2n-1]$. Given a point $V$ in the orbit $C_\J$, there exist coefficients $g_1, \dots, g_{2n}$ such that
\[V_{i+m} = \Span_{\bC} (\underbrace{0, \dots, 0}_m, g_1, g_2, \dots, g_{2n-m})\]
 with $g_1 \ne 0$ (see \cite[Theorem~3.10]{FLP22} and \cite[Theorem~4.13]{Pue2022}). 
  Now let $m'$ be the number in $[0,2n-1]$ such that $i+m \equiv_{2n} -(i+m')$; observe that if $m+m' \ge 2n$, the symplectic form between the generators of $V_{i+m}$ and $V_{i+m'}$ trivially vanishes, so we can assume $m'+1 \le 2n-m$. We compute the product and find
 \[\sum_{s=m'+1}^{2n-m} (-1)^{s+1} g_{s-m'} \cdot g_{2n-m+1-s} \, .\]
There is an even number of summands since $m$ and $m'$ have the same parity, and hence $m'+1$ and $2n-m$ have different parity. Moreover, every summand appears twice with opposite signs, so that the sum is zero. We conclude the proof by observing that the union of all maximal cells is dense in $X(1,2n)$ and that $X(1,2n)^{sp}$ is a closed subvariety of $X(1,2n)$.
\end{proof}
\subsection{Interpretation as degeneration of the isotropic Grassmannian}\label{sec:degeneration}
  The construction of $X(k,2n)^{sp}$ mimics the definition of the classical isotropic Grassmannian $Gr(k,2n)^{sp}$, that is, the projective variety of isotropic subspaces of $\bC^{2n}$. Recall from \cite{FLP22} that $X(k,n)$ is a degeneration of $Gr(k,n)$: for $z \in \bC$ define $U_{[n]}(z)$ as the $\Delta_n$-representation with $\bC^n$ on every vertex and $\tau_{1,z}$ on every map, where the matrix of $\tau_{1,z}$ in the standard basis is
\[\tau_{1,z} = \begin{pmatrix}
    0 & 0 & \cdots & 0 & z \\
    1 & 0 & \cdots & 0 & 0 \\
    \vdots & \ddots & \ddots & \vdots & \vdots \\
    0 & 0 & \ddots & 0 & 0 \\
    0 & 0 & \cdots & 1 & 0
\end{pmatrix}.\]
Then $Gr(k,n)$ is isomorphic to $Gr(\underline{k}, U_{[n]}(z))$ for all $z \ne 0$ (for example by taking the vector space on vertex 0). Since $\tau_{1,0} = \tau_{1}$, we get $U_{[n]}(0) = U_{[n]}$, so these varieties form a family $Y$ equipped with a morphism $\pi \colon Y \longrightarrow \bC$. The fibers over nonzero numbers are all isomorphic to the classical Grassmannian, and the fiber over 0 is $X(k,n)$. Now, we replace $n$ by $2n$ and take $k \le n$, in order to realize $X(k,2n)^{sp}$ as a degeneration of $Gr(k,2n)^{sp}$. Observe that $\Omega \tau_{1,z}^t \Omega = \tau_{1,z}$, so we can define a map
\[\sigma_z \colon Gr(\underline{k}, U_{[2n]}(z)) \longrightarrow Gr(\underline{2n-k}, U_{[2n]}(z))\]
as in Proposition~\ref{tau}. For every fiber $\pi^{-1}(z)$, we consider points that satisfy $V \subseteq \sigma_z (V)$ . This condition is linear with $z \in \bC$ and the fibers over nonzero numbers are isomorphic to the isotropic Grassmannian. Thus we have a subvariety $Y^{sp}$ of $Y$ such that $X(k,2n)^{sp}$ is the desired degeneration. Lastly we observe that $G = \Aut_{\Delta_n}(U_{[n]})$ is a degeneration of $GL_n$ if we see it as the special fiber of the family over $\mathbb{C}$ whose fiber over $z\in\mathbb{C}$ is the algebraic group $\Aut_{\Delta_n}(U_{[n]}(z))$.

\begin{example}\label{exple: symplectic 24}
    We compute the poset of symplectic $(2,4)$-juggling patterns, with the combinatorial order inherited by $JP(2,4)$. We write $ij$ for the set $\{i,j\}$, and we describe a juggling pattern $\J$ in the following way:
    \[\begin{matrix}
        &J_0& \\
        J_3&&J_1 \\
        &J_2&
    \end{matrix}\]

      Out of the 33 juggling patterns from $JP(2,4)$, only these 13 are symplectic:
    
    \begin{gather*}
        \begin{vmatrix}
        &12& \\
        14&&23 \\
        &34&
    \end{vmatrix} \begin{vmatrix}
        &24& \\
        13&&13 \\
        &24&
    \end{vmatrix} \begin{vmatrix}
        &13& \\
        24&&24 \\
        &13&
    \end{vmatrix} \begin{vmatrix}
        &34& \\
        23&&14 \\
        &12&
    \end{vmatrix} \\ \hline
    \begin{vmatrix}
        &24& \\
        14&&23 \\
        &34&
    \end{vmatrix} \begin{vmatrix}
        &13& \\
        24&&24 \\
        &34&
    \end{vmatrix} \begin{vmatrix}
        &24& \\
        34&&34 \\
        &24&
    \end{vmatrix} \begin{vmatrix}
        &34& \\
        24&&24 \\
        &13&
    \end{vmatrix} \begin{vmatrix}
        &34& \\
        23&&14 \\
        &24&
    \end{vmatrix} \\ \hline
    \begin{vmatrix}
        &24& \\
        34&&34 \\
        &34&
    \end{vmatrix} \begin{vmatrix}
        &34& \\
        24&&24 \\
        &34&
    \end{vmatrix} \begin{vmatrix}
        &34& \\
        34&&34 \\
        &24&
    \end{vmatrix} \\ \hline
    \begin{vmatrix}
        &34& \\
        34&&34 \\
        &34&
    \end{vmatrix}
    \end{gather*}

The Hasse diagram of $JP(2,4)^{sp}$ is the following (the minimal vertex is on the bottom):
    
    
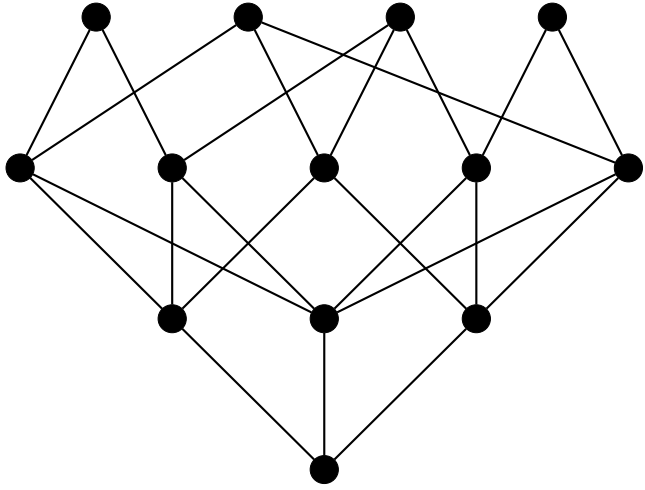
\begin{center}
    \begin{tikzpicture}
     \begin{scope}[every node/.style={circle, draw=black!100, fill=black!100,
     very thin,minimum size=1mm}]
    \node (1) at (0,0) {};
    \node (2) at (-2,2) {};
    \node (3) at (0,2) {};
    \node (4) at (2,2) {};
    \node (5) at (-4,4) {};
    \node (6) at (-2,4) {};
    \node (7) at (0,4) {};
    \node (8) at (2,4) {};
    \node (9) at (4,4) {};
    \node (10) at (-3,6) {};
    \node (11) at (-1,6) {};
    \node (12) at (1,6) {};
    \node (13) at (3,6) {};
\end{scope}

\begin{scope}[every edge/.style= 
              {draw=black,thick}]
\path [-] (1) edge[] (2);
\path [-] (1) edge[] (3);
\path [-] (1) edge[] (4);
\path [-] (2) edge[] (5);
\path [-] (2) edge[] (6);
\path [-] (2) edge[] (7);
\path [-] (3) edge[] (5);
\path [-] (3) edge[] (6);
\path [-] (3) edge[] (8);
\path [-] (3) edge[] (9);
\path [-] (4) edge[] (7);
\path [-] (4) edge[] (8);
\path [-] (4) edge[] (9);
\path [-] (5) edge[] (10);
\path [-] (5) edge[] (11);
\path [-] (6) edge[] (10);
\path [-] (6) edge[] (12);
\path [-] (7) edge[] (12);
\path [-] (7) edge[] (11);
\path [-] (8) edge[] (12);
\path [-] (8) edge[] (13);
\path [-] (9) edge[] (11);
\path [-] (9) edge[] (13);
\end{scope}
    \end{tikzpicture}
    \end{center}
\end{example}

\section{Symplectic Automorphisms}    
\label{sec:automorphisms}
  Let us now discuss the largest subgroup of $G = \Aut_{\Delta_{2n}}(U_{[2n]})$ which preserves the symplectic form. 

\begin{definition}\label{def:symp-automorphism}
    An automorphism $A=(A_i)_{i\in\bZ_{2n}} \in G$ is symplectic if
    \begin{equation}\label{preservesform}
        (A_i(v), A_{-i}(w)) = (v,w)
    \end{equation}
    for all $i \in \bZ_{2n}$, $v \in U_{[2n]}^{(i)}$ and $w \in U_{[2n]}^{(-i)}$. The subgroup of such elements is denoted by $G^{sp}$.
\end{definition}

  Recall that the choice of $\Omega$ corresponds to the choice of an involutive (non trivial) automorphism of $GL_{2n}$, given by:
\[g \longmapsto - \Omega \cdot g^{-t} \cdot \Omega,\]
whose fixed-point subgroup is $Sp_{2n}$. We provide $G$ with a similar automorphism $\sigma_G: G\rightarrow G$: for any $A=(A_i)_{i\in\bZ_{2n} }\in G$ we define
\begin{equation}\label{eqn: involution on automorphism gp}
    \sigma_G(A) \coloneqq \bigl(-\Omega \cdot (A_{-i}^{-t}) \cdot \Omega \bigr)_i \, .
\end{equation}
Observe that two matrices $A$ and $B$ in $GL_{2n}$ with $(A(v), B(w)) = (v,w)$ for all $v,w \in \bC^{2n}$ must satisfy
\[B = -\Omega \cdot (A^{-t}) \cdot \Omega.\]
Hence the $\sigma_G$-fixed subgroup of $G$ is precisely $G^{sp}$. 

\medskip

Recall the morphism $\sigma: X(k,2n)\rightarrow X(2n-k,2n)$ defined in Proposition \ref{tau}.

\begin{lem}\label{sigmas}
The equality
\[\sigma_G(A) \cdot V = \sigma \bigl( A \cdot (\sigma V)\bigr)\]
holds for all $V \in X(k,2n)$ and $A \in G$.
\end{lem}

\begin{proof}

  Since $\bigl(A \cdot (\sigma V)\bigr)_i = A_i \cdot (V_{-i}^\perp)$ we have
\[
    \Bigl(\sigma \bigl(A \cdot (\sigma V)\bigr) \Bigr)_i = \Bigl(  A_{-i} \cdot (V_i^\perp) \Bigr)^\perp.
\]
As in the proof of Proposition~\ref{tau}, we use the fact that given a matrix $M$ and a subspace $W$, there is an inclusion $\Omega M^t \Omega \bigl( (M \cdot W)^\perp \bigr) \subseteq W^\perp$ which is an equality whenever $M$ is invertible. In this case the identity can be rewritten as $(M\cdot W)^\perp = \Omega M^{-t} \Omega (W^\perp)$. We obtain the claim by plugging in $M =A_{-i}$ and $W=V_i^\perp$ since $A$ consists of invertible matrices. 
\end{proof}

\begin{cor}
    The map $R \colon JP(k,2n) \longrightarrow JP(2n-k,2n)$ is order preserving.
\end{cor}

\begin{proof}
    Firstly, notice that $R$ sends the minimal $(k,2n)$-juggling pattern, whose sets are all equal to $\{2n-k+1, \dots, 2n\}$ to the minimal $(2n-k,2n)$-juggling pattern, which is constantly equal to $\{k+1, \dots, 2n\}$. Now let $V = A \cdot p_{\J'}$ for some $A \in G$ and $\J'$ some $(k,2n)$-juggling pattern. By Lemma~\ref{sigmas} we obtain,
    \[\sigma V = \sigma (A \cdot p_{\J'}) =\sigma(A\cdot \sigma(p_{R\J'}))= \sigma_G(A) \cdot p_{R\J'}.\]
    Hence $\sigma$ sends isomorphically the cell $C_{\J'}$ into the cell $C_{R\J'}$, since they are $G$-orbits. Recall that, for another juggling pattern $\J$, the condition $\J \le \J'$, is equivalent to $p_\J \in \overline{C_{\J'}}$. Since $\sigma$ is an isomorphism of varieties, we get $p_{R\J} \in \overline{C_{R\J'}}$
\end{proof}

\subsection{Explicit description of the symplectic automorphisms}
Let $A=(A_i)_{i\in\bZ_{2n}} \in G$.
For any $i\in\bZ_{2n}$, we denote by $\{e_1^{(i)}, \ldots, e_{2n}^{(i)}\}$ the standard basis of  $U_{[2n]}^{(i)}=\bC^{2n}$.
With respect to such a basis, each $A_i$ is a lower triangular matrix 
with nonzero diagonal entries and $A$ is completely determined by the the entries of the first column of each of the $A_i$'s   \cite[Proposition~4.5]{FLP22}. Let us denote the entries of the first column of $A_i$ by $a^{(i)}_j$ for $i\in\bZ_{2n}$ and $j\in [2n]$.


\begin{lem}\label{conditions} We have that
$A\in G^{sp}$ if and only if the following conditions hold true 
\begin{equation}
a^{(i)}_1 a^{(j)}_1 = 1, \quad  i,j\in \bZ_{2n}, \ i+j=1,\\ \label{symp_eq_1}
\end{equation}
\begin{gather}
\sum_{\ell=0}^{r-1} (-1)^{\ell} a_{1+\ell}^{(i)}\,a_{r-\ell}^{(r-i)} = 0\quad i\in\bZ_{2n}, \ r=2,\dots,2n.\label{symp_eq_2}
\end{gather}
\end{lem}
\begin{rem}
 The relation (\ref{symp_eq_2}) is trivial for $r=2i$. 
\end{rem}
\begin{proof}[Proof of Lemma~\ref{conditions}]
By Definition~\ref{def:symp-automorphism}, $A\in G^{sp}$ 
if and only if
\[
(A_i e^{(i)}_{j_1},A_{-i} e^{(-i)}_{j_2})=(e^{(i)}_{j_1},e^{(-i)}_{j_2})
\]
holds for any $i\in \bZ_{2n}$, any $j_1,j_2\in [2n]$. Now the desired equations come from the explicit form of the matrices $A_i$ as in \cite[Proposition~4.5]{FLP22}.  
\end{proof}


\subsection{Lie algebra of the symplectic automorphism group}
  Now we compute the dimensions of $G^{sp}$ and of $X(k,2n)^{sp}$. Let $\fg \coloneqq \text{Lie}(G) = {\rm End}(U_{[2n]})$. These endomorphisms are explicitly described in \cite[Proposition~4.5]{FLP22}. 
  
Let $\bigl( x(a,b) \, \vert \, a \in [2n], b \in \bZ_{2n} \bigr)$ be the $\bC$-basis of $\fg$ acting as
\[
x(a,b) \bigl( e_{1+j}^{(b+j)} \bigr) = e_{a+j}^{(b+j)} \quad \text{for} \ j=0, \dots, 2n-a\]
and as zero when applied to any other basis vector. With respect to the basis $\left\{e^{(i)}_j\mid j\in [2n]\right\}_{i\in\bZ_{2n}}$, the operator $x(a,b)$ is the matrix tuple $(x(a,b)_{i})_{i\in\bZ_{2n}}$ whose $b+j$-th block ($j\in[0,2n-1] $) has $(s,t)$-entry equals to
\[
x(a,b)_{b+j, (s,t)}=
\begin{cases}
    1&\hbox{ if }s-j=a \hbox{ and }t-j=1,\\
    0&\hbox{ otherwise}.
\end{cases}
\]
by \cite[Proposition~4.5]{FLP22}. Notice that in particular the block $x(a,b)_{b+j}$ is the null matrix as soon as $a+j> 2n$.

\smallskip

Recall that the group $G$ is equipped with the automorphism $\sigma_G$
 from \eqref{eqn: involution on automorphism gp}. This induces a Lie algebra automorphism of $\fg$ of order 2:
\[
    \sigma_{\fg} (x_i) \coloneqq (\Omega x_{-i}^t \Omega)_i \, .
\]
\begin{prop}
    The automorphism $\sigma_\fg$ acts on the basis of $\fg$ via:
    \[\sigma_{\fg} \bigl( x(a,b) \bigr) = (-1)^a \cdot x(a,a-b).\]
\end{prop}
\begin{proof}
    Fix $a \in [2n]$, $b \in \bZ_{2n}$ and let $-i = b+j$ for $j \in [0,2n-1]$. If $j \le 2n-a$ then all entries of $x(a,b)_{-i}$ are zero except for a 1 in position $(a+j,1+j)$, hence its transpose has only a 1 in position $(1+j,a+j)$.
    Because $\Omega_{st} = (-1)^{s+1} \delta_{s+t,2n+1}$, for any matrix $A=(A_{st})_{st}$ we obtain
    \[(\Omega A \Omega)_{st} = (-1)^{s+t+1} \cdot A_{2n-s+1, 2n-t+1}.\]
    Thus $\Omega x_{-i}^t \Omega$ has zeros everywhere except a $(-1)^{(1+j)+(a+j)+1}=(-1)^a$ in position $(2n-j,2n-a-j+1)$. If we let $\ell = 2n-a-j \in [0,2n-a]$, then $i = a-b+\ell$ and $\sigma_\fg\bigl(x(a,b)\bigr)_i = (-1)^a\cdot x(a,a-b)_i$.
    If $j \geq 2n-a+1$ then $x(a,b)_{-i}=0$ and $\ell + 2n \ge 2n-a+1$, so $\sigma_\fg \bigl(x(a,b)\bigr)_i = 0 = x(a,a-b)_i$.
\end{proof}
\begin{dfn}\label{def:lie-generators}
    Let $y(a,b) \coloneqq \frac{1}{2} \bigl[ x(a,b)+(-1)^a \cdot x(a,a-b) \bigr]$.
\end{dfn}
\begin{prop}\label{prop:dim-G-sp}
The dimension of $G^{sp}$ is $2n^2 +n$.
\end{prop}
\begin{proof}
The elements from Definition~\ref{def:lie-generators} span $\fg^{\sigma_\fg}$, and since $y(a,b) = (-1)^a \cdot y(a,a-b)$, a basis of $\fg^{\sigma_\fg}$ made of such elements has cardinality $2n^2 +n$. We end the proof by remarking that $\fg^{\sigma_\fg}$ is the Lie algebra of $G^{sp}$, which therefore has the same dimension.
\end{proof}
\subsection{Orbits of symplectic juggling patterns}
\begin{prop}\label{goodcells}
If $p_\J\in X(k,2n)^{sp}$, then \[C_\J \cap X(k,2n)^{sp} = G^{sp}. p_\J \, .\]    
\end{prop}
\begin{proof}
Let $V=(V_i)_i$ be a point in $C_\J\cap X(k,2n)^{sp}$. We want to show that there exists $A=(A_i)_i\in G^{sp}$
such that $V=A.p_\J$. Observe that this $A$ is not unique since $\dim G^{sp} > \dim X(k,2n)^{sp}$. Let $e_j^{(i)}$ denote the $j$-th standard basis vector of the $i$-th copy of $\bC^{2n}$ as vector space of the $\Delta_{2n}$-representation $U_{[2n]}$. The point $p_\J$, seen as a $\Delta_{2n}$-module, is a direct sum of several indecomposables $p_{\J,1}\oplus\dots\oplus p_{\J,s}$ (see \cite[Example~3.3]{FLP22} and \cite[Proposition~3.2]{FLP23a}). For $c \in [s]$, let 
$e_{j_c}^{(i_c)}$ denote the basis vector corresponding to the starting point of the indecomposable $p_{\J,c}$ viewed as a subrepresentation of $U_{[2n]}$.
Then $V$ is completely determined by vectors 
$v_{i_c}\in V_{i_c}$, for $c=1,\dots, s$  such that 
\[
v_{i_c} = \sum_{j=j_c}^{2n} b_{j,i_c} e^{(i_c)}_j,\quad \mathrm{with} \ b_{j_c,i_c}\ne 0.
\]
Since $V\in X(k,2n)^{sp}$, we know that $(V_i,V_{-i})=0$.
The equality $A.p_J=V$ means that $s$ (out of $2n$) first columns start with the numbers $b_{\bullet,\bullet}$. More precisely, 
the following equalities hold for all $c \in [s]$: 
\begin{equation}\label{cells-via-symp-aut}
    a_1^{(i_c-j_c+1)} = b_{j_c,i_c},\quad
a_2^{(i_c-j_c+1)} = b_{j_c+1,i_c},\quad\ldots\ , \quad
a_{2n-j_c+1}^{(i_c-j_c+1)} = b_{2n,i_c}.
\end{equation}

We are left with the following problem: given several first entries of first columns of several matrices $A_i$ satisfying the
orthogonality conditions, 
we have to complete this data to an element $A\in G^{sp}$. We proceed by inductive application of Lemma~\ref{conditions} to prove this claim:

The equations from Lemma~\ref{conditions} satisfy the following properties:
\begin{itemize}
\item $a_1^{\bullet}$ is either fixed by (\ref{cells-via-symp-aut}) and (\ref{symp_eq_1}) or free; 
\item for any pair $i_1\ne i_2\in\bZ_{2n}$ there is a single relation involving $a^{(i_1)}_{\bullet}$ and $a^{(i_2)}_{\bullet}$;
\item the relation
\[
a^{(i)}_1 a^{(r-i)}_r = a^{(i)}_2 a^{(r-i)}_{r-1} - \dots + (-1)^{r} a^{(i)}_r a^{(r-i)}_1
\]
allows one to reconstruct $a^{(r-i)}_r$ starting from $a^{(\bullet)}_{r'}$ with $r'\le r$.
\end{itemize}
Observe that the number of free parameters depends on $\dim\mathrm{Stab}_{G^{sp}}p_\J$.
The procedure of recovering $A\in  G^{sp}$ is as follows. 
Start with $a^{\bullet}_1$: some of them are known from $V$ and all the others are free or recovered via relation (\ref{symp_eq_1}). 
We proceed with $a^{\bullet}_2$ and so on. At each step some coefficients are known from the data given by $V$ as in (\ref{cells-via-symp-aut}). All the others are either free or recovered by using the relations from Lemma~\ref{conditions}.
\end{proof}

\begin{example}
Let $k=n=2$ and let $\J=(J_0,J_1,J_2,J_3)$ be given by $(\{1,2\},\{2,3\},\{3,4\},\{1,4\})$.   
Then $p_\J=p_{\J,1}\oplus p_{\J,2}$, where both indecomposables $p_{\J,1}$ and $p_{\J,2}$ are four-dimensional, $p_{\J,1}$ starts at vertex $0$ and $p_{\J,1}$ starts at vertex $3$ (hence $s=2$, $i_1=0$, $i_2=3$).  
A point $V=(V_i)_i\in C_\J\cap X(k,2n)^{sp}$ is completely determined by two vectors:
\begin{gather*}
v_0 = b_{1,0}e_1^{(0)} + b_{2,0} e_2^{(0)} + b_{3,0} e_3^{(0)} + b_{4,0} e_4^{(0)} \in V_0,\ b_{1,0}\ne 0, \\
v_3 = b_{1,3}e_1^{(3)} + b_{2,3} e_2^{(3)} + b_{3,3} e_3^{(3)} + b_{4,3} e_4^{(3)}\in V_3, \ b_{1,3}\ne 0.
\end{gather*}
Hence one gets $a^{(0)}_j=b_{j,0}$, $a^{(3)}_j=b_{j,3}$ for $1\le j\le 4$.  The elements  $a^{(1)}_j$ and $a^{(2)}_j$ are subject to the following relations:
\begin{gather*}
a_{1}^{(2)}=b_{1,3}^{-1},\ a_{1}^{(1)}=b_{1,0}^{-1}, \ a_{2}^{(2)}b_{1,0}=b_{2,0} b_{1,3}^{-1}, \\ 
a_{1}^{(2)}a_{3}^{(1)} - a_{2}^{(2)}a_{2}^{(1)} + a_{3}^{(2)}a_{1}^{(1)} = 0,\ 
a_{1}^{(1)}b_{4,3} - a_{2}^{(1)}b_{3,3} + a_{3}^{(1)}b_{2,3} -  a_{4}^{(1)}b_{1,3}= 0.
\end{gather*}
One easily sees that this system does have solutions.
\end{example}

\begin{rem}
    Lemma~\ref{conditions} together with Proposition~\ref{goodcells} allows to compute the dimension of the cells $C_\J^{sp} \coloneqq C_\J\cap X(k,2n)^{sp}$. In Lemma~\ref{lem:symplectic-mutations} and Proposition~\ref{symplecticdim} we provide combinatorial formulas.
\end{rem}
\subsection{Orbit structure of the main object}
\begin{thm}\label{thm:group-orbits-are-cells}
The group $G^{sp}$ is a degeneration of the classical group $Sp_{2n}$; in particular, $\dim G^{sp}=2n^2+n$. The group $G^{sp}$ acts on $X(k,2n)^{sp}$ with a finite number of orbits, the orbits are affine cells and are naturally labeled by the symplectic juggling patterns. 
\end{thm}
\begin{proof}
 The degeneration procedure is analogous to  $X(k,2n)^{sp}$ as degeneration of $Gr(k,2n)^{sp}$, as described in Section~\ref{sec:degeneration}. The dimension of $G^{sp}$ is computed in Proposition~\ref{prop:dim-G-sp}. In Proposition~\ref{goodcells}, the symplectic subsets of cells are obtained as $G^{sp}$-orbits of the points corresponding to juggling patterns as defined in Remark~\ref{rem:coordinate-sub-reps}. This explicit description of the symplectic subsets of the $G$-orbits, together with the equations from Lemma~\ref{conditions}, imply that they are affine cells, since this holds for the $G$-orbits by \cite[Theorem~1]{FLP22}. 
\end{proof}

\section{Properties of the symplectic orbits}
\label{sec:orbits}
In this section we examine the poset structures on $JP(k,2n)^{sp}$ and the dimension of the cells $C_\J^{sp} \subset X(k,2n)^{sp}$ for $\J \in JP(k,2n)^{sp}$. The computations are based on the combinatorics of the so called coefficient quivers associated to the juggling patterns.
\subsection{Mutations of coefficient quivers}
Let $B = \{ e_j^{(i)} \colon j \in [2n]\}_{i\in\bZ_{2n}}$, where $e_j^{(i)}$ is as usual the $j$-th standard basis vector in the $i$-th copy of $\bC^{2n}$ as vector space of the $\Delta_{2n}$-representation $U_{[2n]}$. 

\begin{dfn}
    Let $Q(U_{[2n]},B)$ denote the coefficient quiver of $U_{[2n]}$ with respect to the basis $B$. Thus $Q$ has $B$ as vertex set and arrows $e_j^{(i)} \to e_{j+1}^{(i+1)}$ for any $i\in \bZ_{2n}$ and any $j\in[2n-1]$ (see \cite[Definition~2.2, Definition~3.3]{FLP22} for more detail).
\end{dfn}
Note that every $(k,2n)$-juggling pattern $\J \in JP(k,2n)$ can be identified with an appropriate successor-closed subquiver $S_\J$ in $ Q(U_{[2n]},B)$ as follows: the subquiver $S_\J$ contains a vertex $e^{(i)}_j$ if $j \in J_i$; moreover, it contains an arrow of $Q(U_{[2n]},B)$ if its source and target are contained in the vertex set of  $S_\J$. Then $S_\J$ is successor-closed in $Q(U_{[2n]},B)$ ( we write $S_\J \subseteq^{sc} Q(U_{[2n]},B)$ for short), i.e. if a vertex $v \in S_\J$ is a source of an arrow $\alpha: v \to w$ of $Q(U_{[2n]},B)$ then $S_\J$ also contains $w$ (and hence $\alpha$). The above identification defines an isomorphism between $JP(k,2n)$ and the set
\[ SC(k,2n) := \big\{ S \subseteq^{sc} Q(U_{[2n]},B) \  : \ \#S\cap \{e^{(i)}_j \ : \ j \in [2n] \} = k \ \mathrm{for} \ \mathrm{all} \ i \in \bZ_n \big\}. \]

\begin{dfn}\label{defn: mutation}
Two elements $S,S' \in SC(k,2n)$ are connected by a mutation $\mu : S' \to S$ if they differ by the position of a (predecessor closed) segment, i.e.:
\begin{align*}
S'\setminus S &= e^{(i)}_{j} \to e^{(i+1)}_{j+1} \to \dots \to e^{(i+\ell)}_{j+\ell},\\
S\setminus S' &= e^{(i)}_{j+s} \to e^{(i+1)}_{j+1+s} \to \dots \to e^{(i+\ell)}_{j+\ell+s}.
\end{align*}
\end{dfn}
We write $S' \geq^\mu S$ if there is a sequence of mutations $\mu_1, \dots, \mu_r$ such that 
\[ 
S = \mu_r \circ \dots \circ \mu_1(S').
\]
By \cite[Corollary~2.22]{LP23b} this defines a partial order on $SC(k,2n)$. In \cite[Theorem~4.6]{FLP23a} we prove the following statement.
\begin{prop}
    There is an order preserving poset isomorphism between $JP(k,2n)$ with the order induced by cell closures in $X(k,2n)$ and $SC(k,2n)$ with the order induced by mutation sequences.
\end{prop}
\subsection{Maximal symplectic juggling patterns}
The above poset isomorphism allows us to apply the combinatorics of mutations to examine the structure of cell closures in $X(k,2n)$ and its subvariety $X(k,2n)^{sp}$. We say $S' \geq^\mu S$ are \emph{adjacent} if there is no sequence of mutations from $S'$ to $S$ with $r >1$. Analogously we define adjacent juggling patterns.
\begin{lem}\label{problemslemma}
    Let $\J$ be a non-maximal $(k,2n)$-juggling pattern such that $\J \subseteq R\J$. Now let $\J' \ge \J$ be another $(k,2n)$-juggling pattern, adjacent to $\J$. Suppose $\J' \not \subseteq R \J'$. Then there exists $\J'' \ge \J'$, adjacent to $\J'$, such that $\J'' \subseteq R\J''$.
\end{lem}

\begin{proof}
    First we set some notation: given $\J$ a juggling pattern, $x \in [2n]$ and $a \in \bZ_{2n}$, we write $x^{(a)} \in \J$ to indicate $x \in J_a$.
    
    We consider the inverse simple mutation in the coefficient quiver that links the successor-closed subquivers associated to $\J$ and $\J'$: let
    \[x^{(a)} \rightarrow {(x+1)}^{(a+1)} \rightarrow \cdots \rightarrow {(x+\ell)}^{(a+\ell)}\]
    be the elements of the sets of $\J$ that mutate into different ones. There exists an integer $s \ge 1$ such that these numbers are substituted with\begin{equation}\label{newsegment}(x-s)^{(a)} \rightarrow {(x+1-s)}^{(a+1)} \rightarrow \cdots \rightarrow {(x+\ell-s)}^{(a+\ell)}\end{equation}
    and $\J'$ differs from $\J$ only by this change, since they are adjacent. Since this is a mutation, we have the following:
    \begin{itemize}
        \item $x \ge 1+s \ge 2$;
        \item $(x-1)^{(a-1)} \notin \J$;
        \item $(x-s-1)^{(a-1)} \notin \J'$;
        \item $x+\ell \le 2n$, and therefore $\ell \le 2n-2$;
        \item $x+\ell-s+1 \in J_{a+\ell+1} \cap J'_{a+\ell+1}$.
    \end{itemize}
    Recall now that $\J$ is symplectic, so we know that for all $a \in \bZ_{2n}$ and all $y \in J_a$, the element $2n+1-y$ is not in $J_{-a}$. By our assumption, $\J'$ does not have this property, that is, there exist one vertex $b$ and one element $y \in J'_b$ such that $2n+1-b \in J'_{-b}$, and it must be one that appears in segment \eqref{newsegment}. We consider the elements paired to the elements of \eqref{newsegment} by the symplectic form. They again form a segment (not necessarily contained completely in $\J'$):
    \begin{equation}\label{faulty}
    (2n-x-\ell+s+1)^{(-a-\ell)} \rightarrow \cdots \rightarrow (2n-x+s)^{(-a-1)} \rightarrow (2n-x+s+1)^{(-a)}.\end{equation}
    We will call the elements of this segment which are in $\J'$ \emph{problems}, and we know there is at least one, say $(2n-x-j+s+1)^{(-a-j)}$ for some $j \in [0,\ell]$. Then for all $0 \le i \le j$, $2n-x-i+s+1$ is in $J'_{-a-i}$, because $\J'$ is a juggling pattern. So the problems make up a successor closed subsegment of \eqref{faulty}. We assume now that $(2n-x-j+s+1)^{(-a-j)}$ is the leftmost problem, so that it starts a segment in the successor-closed subquiver representing $\J'$. To obtain a new symplectic juggling pattern from $\J'$, we want to apply a mutation to remove this exact segment, again by subtracting the integer $s$. We can do so since $2n-x-j+s+1 \ge s+1$ (recall that $s\ge 1$, $j \le \ell$ and $x+\ell \le 2n$), and it creates a juggling pattern $\J''$ which is symplectic, because any problem with it would be of the form $(2n-x+1+i)^{(-a+i)}$, with $i \ge 2n-j$, and it is not possible since the element paired to it is neither in $\J'$ nor $\J''$, since they were removed with the first mutation \eqref{newsegment}.
\end{proof}
\subsection{Symplectic mutations and closures of symplectic cells}
\begin{cor}
    If $\J'' > \J$ are symplectic and adjacent in $JP(k,2n)$, or they are as in the previous lemma, then $p_\J$ is in the closure of the $G^{sp}$-orbit of $p_{\J''}$ in $X(k,2n)^{sp}$.
\end{cor}

\begin{proof}
    We define the path $V(t), t \in \bC$, as follows: for $a\in\bZ_{2n}$ we set
    \[V(t)_a \coloneqq \Span \Bigl ( \{e_j^{(a)} \, \vert \, j \in J_a \cap J''_a\} \cup \{e_{j-s}^{(a)} + t \cdot e_j^{(a)}  \, \vert \, j \in J_a \text{\textbackslash} J''_a \} \Bigr ) \, ,\]
where $s$ is as in Definition \ref{defn: mutation} if $\J$ and $\J''$ are adjacent, or as in the proof of the previous Lemma if they are not.
It follows from the explicit description of the mutations as in Lemma~\ref{problemslemma} that $V(t)$ is a point in the cell $C_{\mathcal{J}''}$ of $X(k,2n)$ for all $t \in \bC$, since it satisfies the equations describing the cells as computed in the proof of \cite[Theorem~4.13]{Pue2022}.
    
    Now, we show that $V(t)$ is contained in $X(k,2n)^{sp}$. For any $j \in J_a \cap J''_a$, the vector $e_j^{(a)}$ pairs trivially with any other element in $V(t)_{-a}$, so we can take $j \in J_a \text{\textbackslash} J''_a$ and $y \in J_{-a} \text{\textbackslash} J''_{-a}$. Then
    \[(t \cdot e_j^{(a)} + e_{j-s}^{(a)}, t \cdot e_y^{(-a)} + e_{y-s}^{(-a)}) = (t \cdot e_j^{(a)}, e_{y-s}^{(-a)}) + (e_{j-s}^{(a)},t \cdot e_y^{(-a)}) =0 \, .\]
    The expression is zero because the two summands are opposites of each other, since $j-s$ and $y-s$ are respectively in $J''_a \text{\textbackslash} J_a$ and in $J''_{-a} \text{\textbackslash} J_{-a}$.

      Finally, we observe that $V(0)=p_{\J''}$ and the boundary point is $p_{\J}$.
\end{proof}
This motivates the following definition: 
\begin{dfn}
 A symplectic mutation is either a mutation $\mu: S' \to S$ such that both $S$ and $S'$ are symplectic or a pair of two mutations $\mu_2 : S' \to S$ and $\mu_1 : S'' \to S'$ where $S$ and $S''$ are symplectic and $\mu_1, \mu_2$ are as described in Lemma~\ref{problemslemma}.   
\end{dfn}
\begin{rem}
Sequences of symplectic mutations define a partial order on $SC(k,2n)^{sp}$ (i.e. the subset of $SC(k,2n)$ corresponding to the juggling patterns in $JP(k,2n)^{sp}$). Hence it is a natural question to ask if this partial order is induced by the mutation order on $SC(k,2n)$.
\end{rem}

\begin{example}In Example \ref{exple: symplectic 24}, we represented the Hasse diagram of $JP(2,4)^{sp}$ equipped with the combinatorial order. We grouped the elements horizontally in tiers by the dimension of their $G^{sp}$-orbit, from 3 to 0, top to bottom. In this case, the closure of a cell $C_\J^{sp}$ coincides with the union of symplectic cells for lower (symplectic) juggling patterns That is, the combinatorial order of symplectic $(2,4)$-juggling patterns coincides with the closure inclusion order on the set of $G^{sp}$-orbits in $X(2,4)^{sp}$. Otherwise stated,  closure inclusion order on the set of $G^{sp}$-orbits in $X(2,4)^{sp}$ is induced by the closure inclusion order on the set of $G$-orbits in $X(2,4)$. We conjecture that this is the case in general.
\end{example}

\begin{conj}\label{conj:induced-partial-order}
Let  $\J,\J' \in JP(k,2n)^{sp}$ such that $p_\J \in \overline{C_{\J'}} \subset X(k,2n)$. Then $p_\J \in \overline{C_{\J'}^{sp}} \subset X^{sp}(k,2n)$.
\end{conj}
In other words: If $S,S' \in SC(k,2n)^{sp}$ are connected by a sequence of mutations, does a sequence of symplectic mutations connecting $S$ to $S'$ exist?
\begin{lem}\label{lem:symplectic-mutations}
For $\J \in JP(k,2n)^{sp}$ the dimension of $C_\J^{sp}$ equals the number of symplectic mutations starting at $S_\J \in SC(k,2n)$.
\end{lem}
\begin{proof}
In \cite[Corollary~6.5, Theorem~6.15]{LP23a}, we prove that $\dim_\bC C_\J$ equals the number of mutations starting at $S_\J$, that is, mutations of the form $\mu \colon S \to S_\J$. Every mutation is linked to one so-called terminal parameter in the equations describing the cells (see \cite[(5.9), (5.10)]{LP23a}). Now, Lemma~\ref{problemslemma} implies that these parameters are only independent if they belong to different symplectic mutations. Clearly, every symplectic mutation allows to choose one independent parameter. This implies the desired dimension formula.
\end{proof}
\begin{rem}
In particular, every mutation starting at $S_\J$ is either symplectic or part of a symplectic pair as described in Lemma~\ref{problemslemma}.
\end{rem}
\subsection{Top dimensional cells}\label{sec:top-dim-cells}

\begin{prop}\label{symplecticdim}
    The dimension of any top dimensional cell in $X(k,2n)^{sp}$ is $k(2n-k) - \frac{k(k-1)}{2}$.
\end{prop}

\begin{proof}
By \cite[Theorem~3]{FLP22} every top dimensional cell $C_\J$ of $X(k,2n)$ is of dimension $k(2n-k)$. This equals the number of mutations starting at $S_\J$ by \cite[Corollary~6.5, Theorem~6.15]{LP23a}. To apply Lemma~\ref{lem:symplectic-mutations}, we count the symplectic mutations starting at $S_\J$. At each of the $k$-many segments in $S_\J$ of length $2n$, there start $2n-k$-many mutations. But $k-1$ of them make $\J$ non-symplectic. They have to be paired with their correction move as described in Lemma~\ref{problemslemma}. In total there start $k(2n-k)-k(k-1)$ single mutations at $S_\J$ which are symplectic and $\frac{k(k-1)}{2}$ symplectic pairs as described in Lemma~\ref{problemslemma}. This sums up to the desired formula.
\end{proof}

\begin{prop}\label{prop:number-top-dim-cells}
    The number of top dimensional cells in $X(k,2n)^{sp}$ is \[ \frac{(2n)!!}{k! (2n-2k)!!}.\]
\end{prop}
\begin{proof}
By Lemma~\ref{problemslemma}, the top dimensional cells $C_\J^{sp}$ in $X(k,2n)^{2p}$ are the intersection of the top dimensional cells $C_\J\subset X(k,2n)$ with $X(k,2n)^{sp}$, where $\J$ is symplectic. Their number is given by the number of cardinality $k$ subsets of $[2n]$ that contain no pair $(i, \Tilde{i})$. There are $2^k \cdot \binom{n}{k} = \frac{(2n)!!}{k! (2n-2k)!!}$ such sets.
\end{proof}
\begin{thm}\label{thm:top-dim-cells}
The intersection of a cell of $X(k,2n)$ with $X(k,2n)^{sp}$
is either empty or an affine cell. The top dimensional cells of $X(k,2n)^{sp}$ are obtained from top dimensional cells of $X(k,2n)$.
The dimension of $X(k,2n)^{sp}$ is equal to the dimension of the classical symplectic Grassmannian $\mathrm{Gr}(k,2n)^{sp}$ and the number of
irreducible components of $X(k,2n)^{sp}$ is equal to the Euler characteristic of $\mathrm{Gr}(k,2n)^{sp}$.
\end{thm}
\begin{proof}In Theorem~\ref{thm:group-orbits-are-cells} the intersection of the cell $C_\J$ with the symplectic subvariety is described as the $G^{sp}$-orbit of the point $p_\J$ corresponding to a symplectic juggling pattern $\J$. In particular, it is an affine cell. If $\J$ is not symplectic, then $\J\not\subseteq R\J$ and hence there is at least one $i\in\bZ_{2n}$ and an $a\in [2n]$ such that $a\in J_i$ and $2n-a+1\in J_{-i}$. By the description of the $G$-orbits in $X(k,2n)$ from \cite{FLP22}, we know that every point $
(V_s)_s\in C_\J$  is such that the space 
$V_i$  contains a vector of the form   $v=e^{(i)}_a +\sum_{s=a+1}^{2n}c_s e^{(i)}_s$, while
$V_{-i}$ contains a vector of the form $w=e^{(-i)}_{2n+1-a}+ \sum_{s=2n+2-a}^{2n}d_s e^{(-i)}_s$ for some complex numbers $c_s, d_s \in\bC$. It is now clear that $(v,w)\neq 0$ and hence $(V_s)_s\not\in X(k,2n)^{sp}$.
For the maximal symplectic juggling patterns, the dimension of the corresponding cell is computed in Proposition~\ref{symplecticdim}. It equals the dimension of $\mathrm{Gr}(k,2n)^{sp}$.
The irreducible components are obtained as closure of the top dimensional cells by Lemma \ref{problemslemma}. Hence, their number matches the Euler characteristic of $\mathrm{Gr}(k,2n)^{sp}$ by Proposition~\ref{prop:number-top-dim-cells} (see \cite{B03} where a similar property holds in a more general setup).
\end{proof}

\section{Affine Grassmannians and flag varieties}
\label{sec:affine}
\subsection{The type A case} Here we briefly recall some background on type A affine  Grassmannians and flag varieties (more details can be found, for example, in \cite{Pue2022} or in \cite{Kum02} in much greater generality). 
Let us fix a positive integer number $N$.
Let $V$ be an $N$-dimensional complex vector space with a basis $v_1,\dots,v_N$. 
We consider the set of
lattices inside the space $V[t,t^{-1}]=V\otimes \bC[t,t^{-1}]$, i.e  $t$-invariant subspaces $L$ such that for some $m\in \bZ_{>0}$ and $c\in \bZ$ one has  
\begin{equation}\label{eqn: lattices}
    t^m V[t]\subset L \subset t^{-m}V[t],\qquad  \dim L/ t^m V[t] = mN  + c.
\end{equation}
We denote by $\mathcal{AG}_{N,c}$ the set of lattices satisfying \eqref{eqn: lattices} for a fixed $c\in\bZ$.
For example, one has a distinguished lattice $\mathring{L}_c\in \mathcal{AG}_{N,c}$ defined as follows. Let us write $c=Nd+r$, where $d\in\bZ$ and $r=0,\dots,N-1$. Then
\begin{equation}
\mathring{L}_c = t^{-d}\left(V[t] \oplus \Span\{v_1t^{-1}, \dots, v_rt^{-1}\}\right).
\end{equation}
The affine Grassmannians $\mathcal{AG}_{N,c}$ are endowed with a structure of infinite-dimensional ind-varieties and the multiplication by $t^{-1}$ induces a natural isomorphism 
$\mathcal{AG}_{N,c}\simeq \mathcal{AG}_{N,c+N}$. These ind-varieties are equipped with a transitive action of the affine Kac-Moody Lie group $\widehat{SL}_N$. 
One has $\mathcal{AG}_{N,c}\simeq \widehat{SL}_N/P_c$ for $c=0,\dots,N-1$, where 
$P_c$ is the parahoric subgroup corresponding to the fundamental weight $\Lambda_c$. The disjoint union $\sqcup_{c=0}^{N-1} \mathcal{AG}_{N,c}$ is identified with the affine Grassmannian for the affine $GL_N$ group. 

The affine flag variety $\mathcal{AF}_N$ sits inside the product $\prod_{c\in \bZ} \mathcal{AG}_{N,c}$ and consists of collections $(L_c)_{c\in\bZ}$ such that $L_c\subset L_{c+1}$
and $L_{c+N}=t^{-1}L_c$. The ind-variety $\mathcal{AF}_N$ is isomorphic to the quotient
$\widehat{SL}_N/\mathbb{I}$, where $\mathbb{I}$ is the Iwahori subgroup defined
as the preimage of the Borel subgroup $B\subset SL_N$ under the $t=0$ evaluation map $SL_N[t]\to SL_N$.  
We note that $\mathcal{AF}_N$ contains the distinguished point $\mathring{L}=(\mathring{L}_c)_{c\in\bZ}$.

\subsection{Symplectic version} 
Assume that $N=2n$ is even. We endow the space $V$ with a skew-symmetric nondegenerate bilinear form defined by $(v_i,v_{j})=(-1)^{i+1}\delta_{i+j,N+1}$.
The form induces a skew-symmetric nondegenerate form on $V[t,t^{-1}]$ given by 
$(vt^r, ut^s) = (v,u)\delta_{r+s,-1}$. For example, the lattice $\mathring{L}_0$
is Lagrangian with respect to this form.

The symplectic affine flag variety $\mathcal{AF}^{{\mathfrak{sp}}}_{2n}$ is the subvariety
of the type $A$ affine flag variety $\mathcal{AF}_{2n}$ defined by
\begin{equation}
(L_c)_{c\in\bZ}\in \mathcal{AF}^{{\mathfrak{sp}}}_{2n}\ \text{ if }\ 
(L_c)_{c\in\bZ}\in \mathcal{AF}_{2n} \text{ and }  L_{-c} = L_c^{\perp}.
\end{equation}
In other words, if we consider a degree two automorphism $\sigma$ of 
$\mathcal{AF}_{2n}$ defined by $\sigma.(L_c)_{c\in\bZ} = (L_{-c}^\perp)_{c\in\bZ}$, then $\mathcal{AF}^{{sp}}_{2n}$ is the 
set of $\sigma$-fixed points. 
Note that the symplectic affine flag variety contains the distinguished point $\mathring{L}=(\mathring{L}_c)_{c\in\bZ}$. 

Let $Sp_{2n}\subset SL_{2n}$ be the Lie group preserving the above defined skew-symmetric form on $V$.
Let $B^{sp} \coloneqq B \cap Sp_{2n}$ be the Borel subgroup of $Sp_{2n}$, consisting of symplectic upper triangular matrices (or rather, one needs the intersection with a Borel subgroup closed under the involution on $SL_{2n}$ given by the chosen symplectic form). We denote by $\widehat{Sp}_{2n}\subset \widehat{SL}_{2n}$ the symplectic affine group. As in type $A$, the Iwahori subgroup $\mathbb{I}^{sp}$ is the preimage of the Borel subgroup $B^{sp}\subset Sp_{2n}$ under the evaluation map $Sp_{2n}[t]\to Sp_{2n}$, $t\mapsto 0$. 
Then $\mathcal{AF}^{sp}_{2n}\simeq \widehat{Sp}_{2n}/\mathbb{I}^{sp}$.

\subsection{Embeddings of quiver Grassmannians}

As usual, for $j\in [2n]$, we denote by $e_j$ the $j$-th standard basis vector of $\bC^{2n}$; moreover, whenever we identify $\bC^{2n}$ with with $U_{[2n]}^{(i)}$, that is the $i$-th space of our $\Delta_{2n}$-representation $U_{[2n]}$, we will use the notation $e_1^{(i)}, \ldots, e_{2n}^{(i)}$ for the standard basis vectors.
\smallskip

\noindent For $k=1,\dots,n$ consider the juggling variety $X(k,2n)$.
We recall the embedding  $\varphi: X(k,2n) \to \mathcal{AF}_{2n}$ from \cite[Section~6.1]{FLP23a}. Given a point $U=(U_i)_{i=0}^{2n-1}$ we describe the components 
$(\varphi U)_c$ of the point $\varphi U\in \mathcal{AF}_{2n}$ below. 

Let us define a family $\eta_{j,d}$ of embeddings $\bC^{2n}\to V[t,t^{-1}]$.
The embeddings are labeled by $d\in\bZ$, $j\in[2n]$;  $\eta_{j,d}$ is defined by 
\begin{align*}
e_{2n} \mapsto  v_jt^d,\ \ \,
e_{2n-1} &\mapsto  v_{j+1}t^d, \dots , \ e_{j} \mapsto  v_{2n}t^d, \\
e_{j-1} \mapsto  v_1t^{d-1}, \
e_{j-2} &\mapsto  v_{2}t^{d-1}, \dots, \, 
e_1 \mapsto  v_{j-1}t^{d-1}.
\end{align*}
Now we define $\varphi: X(k,2n) \to \mathcal{AF}_{2n}$ in the following way:
(recall that  $\mathring{L}_{-n}=tV[t]\oplus\Span\{v_1t^0,\dots,v_nt^0\}$):
\begin{gather*}
(\varphi U)_0 = \mathring{L}_{-n}\oplus  \eta_{n+1,0} U_0,\\
(\varphi U)_1 = \mathring{L}_{-n+1}\oplus 
\eta_{n+2,0} U_1,\\
\dots\dots\dots\dots\dots\dots\\
(\varphi U)_{n-1} = \mathring{L}_{-1}\oplus 
\eta_{2n,0} U_{n-1},\\
(\varphi U)_n = \mathring{L}_{0}\oplus 
\eta_{1,-1} U_{n},\\
\dots\dots\dots\dots\dots\dots\\
(\varphi U)_{2n-1} = \mathring{L}_{n-1}\oplus 
\eta_{n,-1} U_{2n-1}.
\end{gather*}
We also set $(\varphi U)_{c+2n} = t^{-1} (\varphi U)_{c}$.

\begin{rem}
One easily sees (see \cite[Lemmas~6.2 and 6.4]{FLP23a}) that  $(\varphi U)_c\in \mathcal{AG}_{2n,c}$ and that $(U_i)_i\in X(k,N)$ implies $\varphi(U)\in\mathcal{AF}_{2n}$.   
\end{rem}

\begin{example}
Let us take the point $(U_i)_{i=0}^{2n-1}\in X(k,2n)$ given by $U_i=\Span\{e^{(i)}_{2n-k+1},\dots,e^{(i)}_{2n}\}$ for all $i\in\bZ_{2n}$. Then $\varphi U = \mathring{L}$.    
\end{example}

Now we are ready to formulate the main result of this section.
\begin{prop}
For any $k=1,\dots,n$ the image $\varphi(X(k,2n)^{sp})$ belongs to 
$\mathcal{AF}^{sp}_{2n}$ seen inside $\mathcal{AF}_{2n}$.
\end{prop}
\begin{proof}
Let us take a point $(U_i)_i\in X(k,2n)^{sp}$.
We need to show that the condition $(U_i, U_{-i})=0$ implies $(\varphi U)_j\subset (\varphi U)_{-j}^\perp$.
Let us start with $j=0$. 
By definition
\[
(\varphi U)_0 = \mathring{L}_{-n} \oplus \eta_{n+1,0} U_0. 
\]
Since $\mathring{L}_{-n}=tV[t]\oplus \Span\{v_1t^0,\dots,v_nt^0\}$ and the image of $\eta_{n+1,0}$ is the subspace
$\Span\{v_{n+1}t^{0},\dots,v_{2n}t^{0},v_1t^{-1},\dots,v_nt^{-1}\}$, we conclude that 
\[(\mathring{L}_{-n},\mathring{L}_{-n} \oplus \eta_{n+1,0} U_0)=0\]
(recall that the skew-symmetric residue pairing is given by  $(v_it^a,v_jt^b)=\delta_{a+b,-1}\delta_{i+j,2n+1}(-1)^{i+1}$). We are left to show that $(\eta_{n+1,0} U_0,\eta_{n+1,0} U_0)=0$.
By definition the map 
\[
\eta_{n+1,0}: W \to \Span\{v_{n+1}t^{0},\dots,v_{2n}t^{0},v_1t^{-1},\dots,v_nt^{-1}\}
\]
sends the form on $W$ to the restriction of the residue form 
to the image). Hence the condition $(U_0,U_0)=0$ implies 
$((\varphi U)_0,(\varphi U)_0)=0$.

\smallskip

Now let us take $i=1,\dots,n$ and let us prove that $((\varphi U)_i,(\varphi U)_{-i})=0$. 
It suffices to show that
\[
\left(\mathring{L}_{-n+i}\oplus \eta_{n+i+1,0}U_i, t(\mathring{L}_{n-i}\oplus \eta_{n-i+1,-1}(U_{2n-i}))\right)=0
\]
(recall $(\varphi U)_{-i}=t(\varphi U)_{2n-i}$). All the pairings except for the single term  
$(\eta_{n+i+1,0}U_i, t\eta_{n-i+1,-1}U_{-i})$ vanish by the definition of the residue pairing. The remaining term is zero due to 
the condition $(U_i,U_{-i})=0$.
\end{proof}

We conclude with the following theorem, which is proved exactly as in \cite{FLP22,FLP23a}.

\begin{thm}\label{thm:emb-affine-symp-flag}
The image of $X(k,2n)^{sp}$ inside $\mathcal{AF}_{2n}^{sp}$ is equal to a union of Schubert cells. 
The embedding $\varphi: X(k,2n)^{sp}\to \mathcal{AF}_{2n}^{sp}$ translates the action of the symplectic automorphism group on $X(k,2n)^{sp}$
into the action of the Iwahori sugroup $\mathbb{I}^{sp}$ on $\mathcal{AF}_{2n}^{sp}$.
\end{thm}

\subsection*{Acknowledgements} 
E.F. was partially supported by the ISF grant 493/24.
M.L. acknowledges the MUR Excellence Department Project 2023--2027 awarded to the Department of Mathematics, University of Rome Tor Vergata CUP E83C18000100006, and the 
PRIN2022 CUP E53D23005550006.
A.P. was funded by the Deutsche Forschungsgemeinschaft (DFG, German Research Foundation) — SFB-TRR 358/1 2023 — 491392403. 

M.L and M.M. are members of the network INdAM-G.N.S.A.G.A.

\appendix
\section{Numerical data}
\label{sec:appendix}
In this section we provide the Euler characteristics $\chi_{k,2n}$ and $\chi_{k,2n}^{sp}$ of the varieties $X(k,2n)$ and $X(k,2n)^{sp}$ (the numbers of juggling patterns and of the symplectic juggling patterns) as well as the corresponding Poincar\'e polynomials $P_{k,2n}(t)$ and $P_{k,2n}^{sp}(t)$. For the symplectic setting we use Lemma~\ref{lem:symplectic-mutations}. 

Let $n=1$. Then 
\[
\chi_{1,2} = \chi_{1,2}^{sp}= 3, \
P_{1,2}(t)=P_{1,2}^{sp}(t)=2t + 1.
\]

Let $n=2$. Then
\begin{gather*}
\chi_{1,4} = \chi_{1,4}^{sp} =  15,\\  
P_{1,4}(t) = P^{sp}_{1,4}(t) =  4t^3 + 6t^2 + 4t + 1,\\
\chi_{2,4} = 33,\quad  \chi_{2,4}^{sp} = 13,\\
P_{2,4}(t) = 6t^4 + 12t^3 + 10t^2 + 4t^1 + 1,\\
P_{2,4}^{sp}(t) = 4t^3 + 5t^2 + 3t^1 + 1. 
\end{gather*}                          
 
Let $n=3$. Then   
\begin{gather*}
\chi_{1,6} = \chi_{1,6}^{sp} =  63,\\  P_{1,6}(t) = P^{sp}_{1,6}(t) =  6t^5 + 15t^4 + 20t^3 + 15t^2 + 6t^1 + 1,\\
\chi_{2,6} = 473,\quad  \chi_{2,6}^{sp} = 293, \\
P_{2,6}(t) = 15t^8 + 60t^7 + 110t^6 + 120t^5 + 90t^4 + 50t^3 + 21t^2 + 6t^1 + 1,\\
P_{2,6}^{sp}(t) = 12t^7 + 47t^6 + 81t^5 + 77t^4 + 48t^3 + 21t^2 + 6t^1 + 1,\\
\chi_{3,6} = 883,\quad \chi_{3,6}^{sp} = 79, \\
P_{3,6}(t) =  20t^9 + 90t^8 + 180t^7 + 215t^6 + 180t^5 + 114t^4 + 56t^3 + 21t^2 + 6t^1 + 1,\\
P_{3,6}^{sp}(t) = 8t^6 + 18t^5 + 22t^4 + 17t^3 + 9t^2 + 4t^1 + 1.
\end{gather*}                          

Let $n=4$. Then   
\begin{gather*}
\chi_{1,8} = \chi_{1,8}^{sp} =  255,\\  P_{1,8}(t) = P^{sp}_{1,8}(t) =  8t^7 + 28t^6 + 56t^5 + 70t^4 + 56t^3 + 28t^2 + 8t^1 + 1,\\
\chi_{2,8} = 5281,\quad  \chi_{2,8}^{sp} = 4053, \\
P_{2,8}(t) = 28t^{12} + 168t^{11} + 476t^{10} + 840t^9 + 1050t^8 + 1008t^7 + 784t^6\\ + 504t^5 + 266t^4 + 112t^3 + 36t^2 + 8t^1 + 1,\\
P_{2,8}^{sp}(t) = 24t^{11} + 166t^{10} + 478t^9 + 798t^8 + 904t^7 + 759t^6 + 501t^5 + 266t^4\\ + 112t^3 + 36t^2 + 8t^1 + 1,
\end{gather*} 
\begin{gather*}
\chi_{3,8} = 26799,\quad \chi_{3,8}^{sp} = 7507, \\
P_{3,8}(t) =  56t^{15} + 420t^{14} + 1400t^{13} + 2870t^{12} + 4200t^{11} + 4788t^{10} + 4480t^9\\ + 3542t^8 + 2408t^7 + 1420t^6 + 728t^5 + 322t^4 + 120t^3 + 36t^2 + 8t^1 + 1,\\
P_{3,8}^{sp}(t) = 33t^{12} + 251t^{11} + 757t^{10} + 1319t^9 + 1588t^8 + 1445t^7 +1042t^6\\ + 613t^5 + 297t^4 + 117t^3 + 36t^2 + 8t^1 + 1.
\end{gather*} 
\begin{gather*}
\chi_{4,8} = 44929,\quad \chi_{4,8}^{sp} = 633, \\
P_{4,8}(t) =  70t^{16} + 560t^{15} + 1960t^{14} + 4200t^{13} + 6426t^{12}\\ + 7672t^{11} + 7532t^{10} + 6272t^9 + 4522t^8 + 2856t^7 + 1588t^6 + 776t^5 + 330t^4\\ + 120t^3 + 36t^2 + 8t^1 + 1,\\
P_{4,8}^{sp}(t) = 16t^{10} + 56t^9 + 106t^8 + 131t^7 + 121t^6 + 93t^5 + 59t^4\\ + 31t^3 + 14t^2 + 5t^1 + 1.
\end{gather*} 


\end{document}